\documentclass[12pt]{amsart}
\usepackage[utf8]{inputenc}

\usepackage[english]{babel}
\usepackage{tabularx}
\usepackage{multirow}
\usepackage{amsmath, amssymb, amsthm, amscd,color,comment}
\usepackage{cancel}
\usepackage{cite}
\usepackage{alltt}
\usepackage[dvipsnames]{xcolor}
\usepackage{array}
\usepackage{caption} 
\usepackage{mathtools}
\usepackage{comment}

\usepackage{enumerate}
\usepackage{cancel}
\newtheorem{theorem}{Theorem}[section]
\newtheorem{definition}[theorem]{Definition}

\newtheorem{question}[theorem]{Question}
\newtheorem{remark}[theorem]{Remark}
\newtheorem{lemma}[theorem]{Lemma}
\newtheorem{corollary}[theorem]{Corollary}
\newtheorem{proposition}[theorem]{Proposition}

\def\AA{\mathbb A}

\def\cC{{\mathcal C}^{(\alpha)}}
\def\cD{\mathcal C}

\def\cH{\mathcal H}

\def\a{\alpha}

\def\fqs{\mathbb F_{q^2}}

\def\fq{\mathbb F_q}

\def\deg{{\rm deg}}

\def\a{\alpha}
\def\char{\mbox{\rm Char}}

\title{Intersection of irreducible curves and the Hermitian curve}
\author{Peter Beelen}
\address{Department of Applied Mathematics and Computer Science, Technical University of Denmark, DK-2800, Kongens Lyngby, Denmark}
\email{pabe@dtu.dk}

\author{Mrinmoy Datta}
\address{Department of Mathematics, \newline \indent
Indian Institute of Technology Hyderabad, Kandi, Sangareddy, Telangana, Pin 502285. India}
\email{mrinmoy.datta@math.iith.ac.in}

\author{Maria Montanucci}
\address{Department of Applied Mathematics and Computer Science, Technical University of Denmark, DK-2800, Kongens Lyngby, Denmark}
\email{marimo@dtu.dk}

\author{Jonathan Niemann}
\address{Department of Applied Mathematics and Computer Science, Technical University of Denmark, DK-2800, Kongens Lyngby, Denmark}
\email{jtni@dtu.dk}

\begin{document}

\maketitle

\begin{abstract}
Let $\cH_q$ denote the Hermitian curve in $\mathbb{P}^2$ over $\fqs$ and $\cD_d$ be an irreducible plane projective curve in $\mathbb{P}^2$ also defined over $\fqs$ of degree $d$. Can $\cH_q$ and $\cD_d$ intersect in exactly $d(q+1)$ distinct $\fqs$-rational points? B\'ezout's theorem immediately implies that $\cH_q$ and $\cD_d$ intersect in at most $d(q+1)$ points, but equality is not guaranteed over $\fqs$. In this paper we prove that for many $d \le q^2-q+1$, the answer to this question is affirmative. The case $d=1$ is trivial: it is well known that any secant line of $\cH_q$ defined over $\fqs$ intersects $\cH_q$ in $q+1$ rational points. Moreover, all possible intersections of conics and $\cH_q$ were classified in \cite{DDK} and their results imply that the answer to the question above is affirmative for $d=2$ and $q \ge 4$, as well. However, an exhaustive computer search quickly reveals that for $(q,d) \in \{(2,2),(3,2),(2,3)\}$, the answer is instead negative. 
We show that for $q \le d \le q^2-q+1$, $d=\lfloor(q+1)/2\rfloor$ and  $d=3$, $q \geq 3$ the answer is again affirmative. Various partial results for the case $d$ small compared to $q$ are also provided.
\end{abstract}
%


\section{Introduction}

When working with projective algebraic varieties and applications thereof, it can happen quite regularly that the intersection of such varieties need to be studied. For example, in algebraic coding theory, one often is led to study how many rational points on a variety $V$ can be zeroes of a given class of algebraic functions $f$ on $V$. This amounts to studying how many common rational points the variety $V$ and the hypersurface defined by $f$ can have. This background was the motivation of a question posed by S\o rensen \cite[Page 9]{SoT} concerning the maximum number of common $\fqs$-rational points that a nondegenerate Hermitian surface $\cH_q^{(2)} \subset \mathbb{P}^3$ defined over $\fqs$, can have with a degree $d$ surface $S \subset \mathbb{P}^3$. 
For $d \le q$, S\o rensen conjectured that 
$$|(S \cap \cH_q^{(2)})(\fqs)| \le d (q^3 + q^2 - q) + q + 1$$
and that equality can be obtained only by varieties $S$ that are the union of $d$ planes in $\mathbb{P}^3$ defined over $\fqs$ that are tangent to $\cH_q^{(2)}$, each containing a common line $\ell$ intersecting $\cH_q^{(2)}$ at $q+1$ points.
For $d=1$ this follows directly from classical results on Hermitian varieties \cite{BC,C}. Further, a proof of this conjecture was given for $d=2$ in \cite{Edoukou1}, for $d=3$ in \cite{Beelen_Datta_2020} and for general $d$ in \cite{Beelen_Datta_2021}. In a similar vein, Edoukou studied intersections of a Hermitian threefold $\cH_q^{(3)} \subset \mathbb{P}^4$ and a degree $d \le q$ hypersurface $S \subset \mathbb{P}^4$ and conjectured in \cite{Edoukou2} that 
$$|(S \cap \cH_q^{(3)})(\fqs)| \le d (q^5 + q^2) + q^3 + 1$$
as well as that equality only occurs in case $S$ is a union of $d$ hyperplanes all defined over $\fqs$ containing a common plane that intersects $\cH_q^{(3)}$ in a Hermitian curve.
Edoukou proved his conjecture for $d=2$, while recently the conjecture was proved for $d=3$ and $q \ge 7$ in \cite{Datta_Manna}.

It is striking that both in S\o rensen's and Edoukou's conjecture, the (hyper)surface $S$ attaining the bound is conjectured to be highly reducible. A similar conjecture exists for higher dimensional Hermitian varieties. It is easy to overlook the at first sight simpler question on how many $\fqs$-rational points a Hermitian curve $\cH_q^{(1)}$ in $\mathbb{P}^2$ and another plane curve of degree $d \le q$ can have in common. For simplicity we will write $\cH_q:=\cH_q^{(1)}$ from now on. A Hermitian curve can be given by various plane models, but all of the resulting curves are projectively equivalent. Therefore one can talk about \emph{the} Hermitian curve and choose a different convenient model whenever the need arises. We will typically consider the models $Y^{q}Z + YZ^q  = X^{q+1}$ and $X^{q+1} + Y^{q+1} + Z^{q+1} = 0$.  

B\'ezout's theorem immediately implies that the number of intersection points, let alone $\fqs$-rational intersection points, is at most $d(q+1)$. Moreover, equality can be obtained by choosing the degree $d$ curve to be the union of $d$ lines defined over $\fqs$, not tangent to the Hermitian curve, and with a common intersection point $P$ also defined over $\fqs$, but not on the given Hermitian curve. 
It is therefore natural to think that perhaps any degree $d$ plane curve intersecting a given Hermitian curve in $d(q+1)$ $\fqs$-rational points, needs to be reducible, at least if $d>1$. However, in \cite{DDK} a complete classification of all possible intersections of conics and $\cH_q$ was given. From their results, it becomes clear that there exist absolutely irreducible conics that intersect $\cH_q$ in $2(q+1)$ $\fqs$-rational points. This motivates the following question:
\begin{question}\label{question}
Given the Hermitian curve $\cH_q$ in $\mathbb{P}^2$ over $\fqs$ and an irreducible plane projective curve $\cD_d$ in $\mathbb{P}^2$ also defined over $\fqs$ of degree $d$. Assume that $\cH_q$ and $\cD_d$ are distinct curves. Can these curves intersect in exactly $d(q+1)$ distinct $\fqs$-rational points? 
\end{question}

As mentioned previously, B\'ezout's theorem immediately implies that $\cH_q$ and $\cD_d$ intersect in at most $d(q+1)$ points (and typically in exactly $d(q+1)$ points), but the point of Question \ref{question} is that we want all points to be $\fqs$-rational. 

We will see in the following sections, that for many $d \le q^2-q+1$, the answer to Question \ref{question} is affirmative. The case $d=1$ is trivial: it is well known that any secant line of $\cH_q$ defined over $\fqs$ intersects $\cH_q$ in $q+1$ rational points. Moreover, we already pointed out that all possible intersections of conics and $\cH_q$ were classified in \cite{DDK} and that their results imply that the answer to Question \ref{question} is affirmative for $d=2$ and $q \ge 4$. However, an exhaustive computer search quickly reveals that for $(q,d) \in \{(2,2),(3,2),(2,3)\}$, Question \ref{question} has a negative answer. 

In this article, we will show for several degrees that Question \ref{question} has an affirmative answer. Large values of $d$ (with respect to $q$) are analyzed in Section \ref{section_d_large}, while the case $d$ small is studied in Section \ref{sec:degree_small}. More precisely, we show that Question 1.1 has an affirmative answer for $q \le d \le q^2-q+1$ in Subsection \ref{sec:degree_large}, for $d=\lfloor(q+1)/2\rfloor$ in Subsection \ref{sec:degree_q_halve} and in Section \ref{sec:degree_small} for $d=3$. Section \ref{sec:degree_small} also contains various partial results in other cases where $d$ is small compared to $q$.

\section{Results for large $d$.} \label{section_d_large}

In this section, we analyze Question 1.1 in the case where $d$ is large. More precisely, we will prove that Question 1.1 has a positive answer when $q  \leq d \leq q^2 - q+1$. Moreover, we will show that for $d=\lfloor(q+1)/2\rfloor$ there exists a plane absolutely irreducible curve of degree $d$ intersecting the Hermitian curve at exactly $d (q+1)$ many distinct $\fqs$-rational points.

\subsection{Degree $d \ge q$.}\label{sec:degree_large}
Here, we show the existence plane absolutely irreducible curve of degree $d$ intersecting the Hermitian curve at exactly $d (q+1)$ many distinct $\fqs$-rational points. We remark that, in view of Bezout's Theorem and the fact that the Hermitian curve contains exactly $q^3 + 1$ many $\fqs$-rational points, Question 1.1 is interesting in the case when $d \le q^2 - q + 1$. 
Thus, in the remainder of this section, we may assume that $q \le d \le q^2-q+1$. It is well-known how two distinct plane Hermitian curves in the plane can intersect each other. In \cite{DD}, it is shown that two distinct Hermitian curves can intersect each other in $(q+1)^2$ many distinct $\fqs$-rational points provided $q \ge 3$. In particular, Question \ref{question} has an affirmative answer for $d=q+1$. In the next theorem, we give a different approach that turns out to work not only for $d=q+1$ but for all $d$ between $q+1$ and $q^2-q$. For our convenience, in this subsection, we will assume that $\cH_q$ is defined by the equation $Y^qZ+YZ^q=X^{q+1}$.

\begin{theorem}\label{thm:d_large}
If $q+1 \le d \le q^2-q$, then there exists an absolutely irreducible curve $\cC_d$ of degree $d$ intersecting $\cH_q$ in exactly $d(q+1)$ distinct $\fqs$-rational points.     
\end{theorem}
\begin{proof}
As mentioned above, let $\cH_q$ be given by the equation $Y^qZ+YZ^q=X^{q+1}$. 
Define $S=\{b \in \fqs \mid b^q+b \neq 0\}$. It is evident that $|S| = q^2-q$.  Note that $b \in S$ if and only if the line $Y - bZ = 0$ passing through $[1:0:0]$ is a secant line to $\cH_q$. Fix a positive integer $d$ with $q+1 \le d \le q^2 - q$. 
 Choose any $d$ distinct elements from $S$, say $b_1,\dots,b_d$ and consider the curve $\cC_d$ given by the equation $$(X^{q+1}-Y^qZ-YZ^q)Z^{d-q-1}=\alpha\prod_{i=1}^d (Y-b_iZ),$$ with $\alpha \in \fqs\setminus \{0\}$. As mentioned above, each line $\ell_i$ given by the equation $Y - b_i Z = 0$ is a secant line and consequently intersects the Hermitian curve at exactly $q+1$ distinct $\fqs$-rational points. Moreover, the point $[1:0:0]$ that is common to $\ell_1, \dots, \ell_d$ does not lie on the Hermitian curve. Consequently, the curve given by the equation $\prod_{i=1}^d (Y-b_iZ) = 0$ which is a union of $d$ lines $\ell_1, \dots, \ell_d$ intersects the Hermitian curve at exactly $d(q+1)$ many distinct $\fqs$-rational points. It follows that $\cC_d$ and $\cH_q$ intersect precisely in the $d(q+1)$ distinct $\fqs$-rational points of $\cH_q$ lying on one of the $d$ secant lines $\ell_1, \dots, \ell_d$. 

We now show that there exists $\alpha \in \fqs$ such that the curve $\cC_d$ is absolutely irreducible. To this end, we choose $\alpha$ such that $\alpha\prod_{i=1}^d (-b_i)$ is a primitive element of $\fqs$. This is possible, since $\prod_{i=1}^d (-b_i) \neq 0$ as follows from the defining condition of $S$, i.e., from $b_i^q+b_i \neq 0$. Since $\alpha\prod_{i=1}^d (-b_i)$ does not have an $e$-th root in $\fqs$, for any  $e \mid q+1$, it follows from the theory of Kummer extensions (see \cite[Proposition 3.7.3]{Stichtenoth}) that the polynomial  $(X^{q+1}-0^qZ-0Z^q)Z^{d-q-1}-\alpha\prod_{i=1}^d (0-b_iZ)$ is irreducible over $\fqs$.
Suppose, if possible, that $\cC_d$ is not absolutely irreducible. Then all its $\fqs$-rational points lie on the intersection of the absolutely irreducible components of $\cC_d$. In particular, all the $\fqs$-rational points on $\cC_d$ would be singular points. It is easily checked that $\cC_d$ at most $d$ singular points if $d>q+1$ and at most $d+q$ singular points if $d=q+1$ over an algebraic closure of $\fqs$. However, since $\cC_d$ intersects $\cH_q$ at exactly $d(q+1)$ distinct $\fqs$-rational points, it must admit at least $d(q+1)$ many distinct $\fqs$-rational points. This leads to a contradiction since $d + q < d(q+1)$.  Thus $\cC_d$ must be absolutely irreducible. 
\end{proof}

In view of Theorem \ref{thm:d_large}, we have settled Question \ref{question} affirmatively for $d \ge q+1$, except for $d=q^2-q+1$. Before we produce a curve of degree $q^2 - q + 1$ with the desired property, we just remark that such a curve should contain all the $\fqs$-rational points on the Hermitian curve, even though the Hermitian curve is not an irreducible component of the curve. As in the previous theorem, we will still use the Hermitian curve given by the equation $Y^qZ+YZ^q=X^{q+1}$.

\begin{theorem}
Let $\cD_{q^2-q+1}$ be the curve defined over $\fqs$ given by the equation
$$X \left( (Y^q+YZ^{q-1})^{q-1}-Z^{q^2-q} \right)+X^{q+1}Z^{q^2-2q}-Y^qZ^{q^2-2q+1}-YZ^{q^2-q}=0.$$    
Then $\cD_{q^2-q+1}$ is an absolutely irreducible curve of degree $q^2-q+1$ intersecting the Hermitian curve 
in exactly $q^3+1$ distinct $\fqs$-rational points.
\end{theorem}
\begin{proof}
Viewing the line given by $Z=0$ as the line at infinity, we see that $\cH_q$ has exactly one point at infinity. As a consequence, the Hermitian curve $\cH_q$ has exactly $q^3$ many $\fqs$-rational points on the affine plane given by $Z = 1$.   
Observe that
\begin{eqnarray*}
X^{q^2}-X & = & X \left((X^{q+1})^{q-1}-1\right)\\  
 & \equiv & X \left( (Y^q+Y)^{q-1}-1 \right) \pmod{X^{q+1}-Y^q-Y}\\
 & \equiv & X \left( (Y^q+Y)^{q-1}-1 \right)+X^{q+1}-Y^q-Y \pmod{X^{q+1}-Y^q-Y}.
\end{eqnarray*}
Since the polynomial $X^{q^2}-X$ vanishes at all the $\fqs$-rational points in the affine plane, in particular, it vanishes at all the $\fqs$-rational affine points of $\cH_q$. By the above congruence, we see that the polynomial $X \left( (Y^q+Y)^{q-1}-1 \right)+X^{q+1}-Y^q-Y \pmod{X^{q+1}-Y^q-Y}$ also vanishes at all the $\fqs$-rational affine points on the Hermitian curve. Homogenizing this polynomial, we obtain exactly the defining equation of the curve $\cD_{q^2-q+1}$. Observing that $[0:1:0]$ is also on $\cD_{q^2-q+1}$, we conclude that $\cD_{q^2-q+1}$ contains all $q^3+1$ distinct $\fqs$-rational points of $\cH_q$. Thus $\cD_{q^2-q+1}$ and $\cH_q$ intersect in exactly $q^3+1$ distinct $\fqs$-rational points. It is clear that $\cD_{q^2-q+1}$ is a curve of degree $q^2-q+1$.

We now prove that $\cD_{q^2-q+1}$ is an absolutely irreducible curve by showing that the polynomial $X \left( (Y^q+Y)^{q-1}-1 \right)+X^{q+1}-Y^q-Y$ is absolutely irreducible. Define $T=Y^q+Y$. First, we claim that the polynomial $X(T^{q-1}-1)+X^{q+1}-T$ is absolutely irreducible. Suppose that $x,t$ are transcendental elements over $\fqs$ satisfying $x(t^{q-1}-1)+x^{q+1}-t=0$. Then writing $\overline{t}=x/t$ and $\overline{x}=1/x$, we have
$$(1-\overline{x})^q\overline{t}^{q-1}-\overline{x}^q \overline{t}^{q-2}+\overline{x}=0.$$ Hence by Eisenstein's criterion, the zero of $\overline{x}$ is totally ramified in the function field extension $\fqs(x,t)/\fqs(x)=\fqs(\overline{x},\overline{t})/\fqs(\overline{x})$ with ramification index $q-1$. Thus the polynomial $X(T^{q-1}-1)+X^{q+1}-T$ is absolutely irreducible. 
As a consequence, the function field extension $$[\fqs(x,t) : \fqs(t)] = q+1 \ \ \text{and} \ \ [\fqs(x,t) : \fqs(x)] = q-1.$$
Now let $y$ satisfy $y^q+y=t$. Observe that $\fqs(x,y)$ is the composite of the function fields $\fqs(t,y)$ and $\fqs(t,x)$ over $\fqs(t)$. Furthermore, the extension $\fqs(t,y)/\fqs(t)$ is an Artin-Schreier extension in which the pole of $t$ is totally ramified (see \cite[Proposition 3.7.10]{Stichtenoth}). In particular, $[\fqs(t,y) : \fqs(t)] = q$. Since the extension degrees $[\fqs(t,y):\fqs(t)]=q$ and $[\fqs(t,x):\fqs(t)]=q+1$ are relatively prime, we see that $[\fqs(t, y) : \fqs(t)] = q(q+1)$. Consequently, $[\fqs(x,y):\fqs(x,t)]=q$, which implies  $$[\fqs(x,y):\fqs(x)]=[\fqs(x,y):\fqs(x,t)]\cdot [\fqs(x,t):\fqs(x)]=q(q-1).$$ Thus the polynomial $X \left( (Y^q+Y)^{q-1}-1 \right)+X^{q+1}-Y^q-Y$ is absolutely irreducible.
\end{proof}

We now show that in the case when $d=q$, Question 1.1 has a positive answer. Similar to the case when $d=q^2-q+1$, the main idea is to start by considering affine $\fqs$-rational points on the Hermitian curve $\cH_q$, given by the equation $Y^qZ+YZ^q=X^{q+1}$.

\begin{theorem}\label{thm:degq}
For $q > 2$ and $\alpha \in \fqs \setminus \fq$, the curve $\cC_q$ of degree $q$ given by the equation $$Y^q+YZ^{q-1}=(\alpha+\alpha^2)X^q-\alpha^3X^{q-1}Z+X^2Z^{q-2}-(\alpha+\alpha^2)XZ^{q-1}-\alpha^3Z^q$$ is absolutely irreducible and it intersects $\cH_q$ in $q(q+1)$ distinct $\fqs$-rational points.
\end{theorem}
\begin{proof}
Let $\alpha_1,\dots,\alpha_{q+1} \in \fqs$ be distinct. Observe that $\cH_q$ and the reducible curve with affine equation $f(X):=\prod_{i=1}^{q+1} (X-\alpha_i)$ intersect in $q(q+1)$ affine $\fqs$-rational points, since the lines with equation $X=\alpha_i Z$ are all secant lines through the point $(0:1:0) \in \cH_q$. Now observe that $f(X)$ can be written in the form $f(X)=X^{q+1}-g(X)$, where $\deg g(X) \ \le q$. The projective curve $\cC_q$ given by the equation $Y^q+YZ^{q-1}=Z^qg(X/Z)$ therefore intersects $\cH_q$ in at least $q(q+1)$ distinct affine $\fqs$-rational points. Using B\'ezout's theorem, we conclude that $\cC_q$ and $\cH_q$ intersect in exactly $q(q+1)$ rational points. Now we choose specifically $f(X)=(X^{q-1}-1)(X-\alpha)(X-\alpha^2)$ with $\alpha \in \fqs \setminus \fq$. Then $g(X)=(\alpha+\alpha^2)X^q-\alpha^3X^{q-1}+X^2-(\alpha+\alpha^2)X-w^3.$ We claim that for $q>2$ and this choice of $g(X)$, the polynomial $Y^q+YZ^{q-1}-Z^qg(X/Z)$ is absolutely irreducible. First, observe that the extension $\fqs(x,y)/\fqs(x)$ defined by $y^q+y=g(x)$ is an Artin-Schreier extension. Introducing the variable $z=y+(\alpha+\alpha^2)^q x$, we see that $\fqs(x,y)=\fqs(x,z)$ and $z^q+z=h(x)$, where $\deg h(x)=q-1$. Hence by the theory of Artin-Schreier extensions, the pole of $x$ is totally ramified in the extension $\fqs(x,z)/\fqs(x)$. This implies that the polynomial $Y^q+Y-g(X)$ is absolutely irreducible and so is its homogenization $Y^q+YZ^{q-1}-Z^qg(X/Z)$.   
\end{proof}

\subsection{Degree $d = \lfloor \frac{q+1}{2} \rfloor$}\label{sec:degree_q_halve}
As the proof of Theorem \ref{thm:degq} shows, the curve given by the equation $$\displaystyle{Y^q+YZ^{q-1}=\frac{\prod_{i=1}^{q+1} (X-\alpha_iZ)-X^{q+1}}{Z}}$$ intersects $\cH_q$ in $q(q+1)$ distinct $\fqs$-rational points. However, this curve need not be irreducible for all possible choices of  $\{\alpha_1, \dots, \alpha_{q+1}\}$. We now use this to our advantage to get the first result of this section.

\begin{corollary}\label{cor:q/2}
Suppose $q$ is even and let $\alpha\in \fqs$ be an element satisfying  $\alpha^q+\alpha=1$. Then the curve $\cC_{q/2}$ with equation $$(Y+\alpha^qX)^{q/2}+\cdots+(Y+\alpha^qX)^2Z^{q/2-2}+(Y+\alpha^qX)Z^{q/2-1}=XZ^{q/2-1}$$ is absolutely irreducible, and it intersects the Hermitian curve $\cH_q$ in $q(q+1)/2$ distinct $\fqs$-rational points.
\end{corollary}
\begin{proof}
First, consider the curve with affine equation $$Y^q+Y=X^{q+1}-(X^q-X)(X+\alpha).$$ As shown in Theorem \ref{thm:degq}, it intersects the Hermitian curve $\cH_q$ in exactly $q(q+1)$ distinct $\fqs$-rational points. Upon simplification, the equation of the curve becomes $Y^q+Y=\alpha X^q+X^2+\alpha X.$ Using the equalities $\alpha^{q^2}=\alpha$ and $\alpha^q+\alpha=1$, we have
\[(Y+\alpha^qX)^q+(Y+\alpha^qX)=X^2+(\alpha^q+\alpha)X=X^2+X.\]
Rewriting the right hand side of the above equation, we get
\[(Y+\alpha^qX)^q+(Y+\alpha^qX)=A^2+A,\] where 
\[A=(Y+\alpha^qX)^{q/2}+\cdots+(Y+\alpha^qX)^2+(Y+\alpha^qX).\]
Putting $\tilde{Y}=Y+\alpha^qX$, the defining polynomial of the curve becomes
$$\tilde{Y}^q+\tilde{Y}+X^2+X=(\tilde{Y}^{q/2}+\cdots+\tilde{Y}^2+\tilde{Y}+X)\cdot (\tilde{Y}^{q/2}+\cdots+\tilde{Y}^2+\tilde{Y}+X+1).$$ 
Since the $X$-degree of the factors is one, the resulting curves are absolutely irreducible. Moreoever, each of the curves intersects $\cH_q$ in $q(q+1)/2$ distinct $\fqs$-rational points. Indeed, B\'ezout's theorem implies that they cannot intersect $\cH_q$ in more points, while together they intersect $\cH_q$ in $q(q+1)$ distinct $\fqs$-rational points. 
\end{proof}

In Corollary \ref{cor:q/2}, we have seen that for even $q$ there exists a curve of degree $d=q/2$ intersecting $\cH_q$ in $d(q+1)$ rational points.  In this section, we show that if $q$ is odd, then there are (nonsingular) plane curves of degree $d = \frac{q+1}{2}$ that intersect the Hermitian curve at exactly $d (q+1)$ points. Hence both for $q$ even and odd, there exists curves of degree $d=\lfloor (q+1)/2 \rfloor$ with the desired property of intersecting $\cH_q$ in $d(q+1)$ rational points. In the following, we will assume that $\cH_q$ is given by the equation $X^{q+1} + Y^{q+1} + Z^{q+1} = 0$.

For $\alpha, \beta \in \fq$ let $C_{\alpha, \beta}$ be the curve defined by equation 
$$\alpha X^{\frac{q+1}{2}} + Y^{\frac{q+1}{2}} + \beta Z^{\frac{q+1}{2}} = 0.$$
Note that if $\alpha\beta \neq 0$ then $C_{\alpha,\beta}$ is an absolutely irreducible curve. It is evident that on the line at infinity, i.e. when $Z = 0$, the curve $C_{\alpha, \beta}$ and the Hermitian curve $\cH_q$ do not intersect. Thus, we restrict our attention to the intersection of $\cH_q$ and $C_{\alpha, \beta}$ in the affine plane given by $Z = 1$. Thus, we look at the following system of two polynomial equations:

$$X^{q+1} + Y^{q+1} + 1 = 0 \ \ \ \text{and} \ \ \ \ \alpha X^{\frac{q+1}{2}} + Y^{\frac{q+1}{2}} + \beta =0.$$
By simply eliminating $Y$, we have 
\begin{equation}\label{half:one}
    (\alpha^2 + 1) X^{q+1} + 2\alpha \beta X^{\frac{q+1}{2}} + (\beta^2 + 1) = 0.
\end{equation}

We claim that there exist $\alpha, \beta \in \fq \setminus \{0\}$, such that the equation \eqref{half:one}, considered as a quadratic expression in $X^{\frac{q+1}{2}}$, has exactly two distinct solutions in $\fq \setminus \{0\}$. Note that, this is possible if and only if the following conditions are satisfied simultaneously:
\begin{enumerate}
    \item[(a)] $\alpha \beta \neq 0$ and $(\alpha^2+1)(\beta^2+ 1) \neq 0$.
    \item[(b)] there exists $\gamma \in \fq\setminus \{0\}$ such that $\gamma^2 + \alpha^2 + \beta^2 + 1 = 0$. 
\end{enumerate}
For a polynomial $f \in \fq[\alpha,\beta,\gamma]$ we denote by $V(f)$ the set of zeros of $f$ in $\AA^3 := \overline{\mathbb{F}}_q^3$. Further, we denote by $V(f)(\fq)$ the set of $\fq$-rational points in $V(f)$. Now the conditions (a) and (b) above are satisfied if and only if the set $$S = V(\gamma^2 + \alpha^2 + \beta^2 + 1) \setminus V(\gamma^2 + \alpha^2 + \beta^2 + 1, \alpha \beta(\alpha^2+1)(\beta^2+ 1)\gamma)$$ has at least one $\fq$-rational point in $\AA^3$. 
It follows from the classification of nonsingular conics in $\mathbb{P}^3$ over finite fields, that $|V(\gamma^2 + \alpha^2 + \beta^2 + 1)(\fq)| = (q+1)^2 - (q+1)$. On the other hand, the affine Bezout's theorem (Lachaud-Rolland) implies that $|V(\gamma^2 + \alpha^2 + \beta^2 + 1, \alpha \beta(\alpha^2+1)(\beta^2+ 1)\gamma)(\fq)| \le 14q$. Combining everything together, we see that when $q > 13$, the set $S$ has at least one $\fq$-rational point. In particular, there exists a pair $(\alpha, \beta) \in \fq \times \fq$, such that the equation \eqref{half:one}, has two distinct solutions in $\fq$ for $X^{\frac{q+1}{2}}$, say $c_1$ and $c_2$. For each of the $c_i$-s, we have $\frac{q+1}{2}$ solutions for $X$ in $\fqs$. It remains to show that for each of the $q+1$ solutions for $X$ thus obtained, we have $\frac{q+1}{2}$ many solutions for $Y$ in $\fqs$. We fix a solution $X_0 \in \fqs$ for $X$. By assumption then $X_0^{\frac{q+1}{2}} \in \fq$ and we have 
$$Y^{\frac{q+1}{2}} = -\beta - \alpha X_0^{\frac{q+1}{2}}.$$
Since the right hand side of the above equation belongs to $\fq$, it follows that the equation has exactly $\frac{q+1}{2}$ solutions in $Y$ in $\fqs$, completing the proof.

\section{Results for small $d$.}\label{sec:degree_small}

In this section, we will restrict our attention to a special family of curves and their intersections with Hermitian curves. To this end, for each $\a \in \fqs \setminus \{0\}$, we define the plane projective curve $\cC_d$  given by the equation $XZ^{d-1} = \alpha Y^d$. Our main goal in this section will be to determine all the values of $\a$ such that $\cC_d$ intersects the Hermitian curve $\cH_q$, given by the equation $X^{q+1} + Y^{q+1} + Z^{q+1} = 0$ at exactly $d(q+1)$ many distinct $\fqs$-ratinal points. We remark that using the current model of the Hermitian curves has several advantages. First of all, it is evident that $\cC_d$ and $\cH_q$ do not admit a point of intersection at the line at infinity given by the equation $Z=0$. This phenomenon reduces our study to find out the values of $\a$ such that the following system of equations
\begin{equation}\label{sys1}
    X^{q+1} + Y^{q+1} + 1 = 0 \ \ \ \ \text{and} \ \ \ \ X = \a Y^d
\end{equation}
has exactly $d(q+1)$ distinct zeroes in $\AA^2 (\fqs)$. Upon substituting $X$ by $\a Y^d$ in the first equation, we see that the system of equations \eqref{sys1} has $d(q+1)$ distinct solutions in $\AA^2 (\fq)$ if and only if the polynomial 
$$f(t) = \a^{q+1} t^{d(q+1)} + t^{q+1} + 1$$
has exactly $d(q+1)$ many distinct roots in $\fqs$. For the brevity of terminology, let us recall the following definition. 

\begin{definition}\label{def:split}\normalfont
Let $\mathbb{K}/\mathbb{F}$ be a field extension and $f(t) \in \mathbb{F}[t]$, with $n = \deg f \ge 1$. We say that \emph{$f(t)$ splits over $\mathbb{K}$} if it has $n$ distinct roots in $\mathbb{K}$. If $\mathbb{K}$ is the smallest field extension of $\mathbb{F}$ such that $f$ splits over $\mathbb{K}$, then we say that $\mathbb{K}$ is the \emph{splitting field of $f$}.
\end{definition}

Thus, we have shown that $|\cH_q \cap \cC_d (\fqs)| = d(q+1)$ if and only if $f(t)$ splits over $\fqs$. The following Lemma further reduces our study to a problem concerning the splitting of certain one-variable polynomials over $\fq$. 

\begin{lemma}\label{lem:split1}
For $\alpha \in \fqs \setminus \{0\}$, the Hermitian curve $\cH_q$ and the curve $\cC_d$ intersect in $d(q+1)$ distinct $\fqs$-rational points if and only if  $\alpha^{q+1}t^d+t+1 \in \fq[t]$ splits over $\fq$. 
\end{lemma}
\begin{proof}
As seen above, the curves $\cH_q$ and $\cC_d$ intersect in $d(q+1)$ distinct $\fqs$-rational points if and only if the polynomial $f(t) = \alpha^{q+1}t^{d(q+1)}+t^{q+1}+1 \in \fq[t]$ splits over $\fqs$. Note that, if an element $\rho \in \fqs$ is a root of $f(t)$, then $\rho^{q+1} \in \fq$ is a root of $g(t)=\alpha^{q+1}t^d+t+1$. This shows that if $f(t)$ has $d(q+1)$ distinct roots in $\fqs$, then $g(t)$ must have $q+1$ distinct roots in $\fq$. This establishes the only if part of the assertion. Conversely, 
if $\sigma \in \fq$ is a root of $\alpha^{q+1}t^d+t+1$, then each of the $q+1$ roots of the polynomial $X^{q+1} - \sigma$  in $\fqs$ is a root of $\alpha^{q+1}t^{d(q+1)}+t^{q+1}+1$. Consequently, if $g(t)$ splits over $\fq$, that is $g(t)$ has $d$ distinct roots in $\fq$, then $f(t)$ must have $d(q+1)$ roots in $\fqs$. This completes the proof. 
\end{proof}

Thus, we are left with the problem of determining the number of $A=\alpha^{q+1} \in \fq \setminus \{0\}$ such that the polynomial $At^d+t+1$ splits. We aim to determine this number using Galois theory. 

\subsection{Galois theory for the polynomial $At^d+t+1$}
For now, let us consider $A$ as a transcendental element over $\fq$ and
let us denote by $F_d$ the splitting field of $At^d+t+1$ over $\fq(A)$. We begin our investigations on $F_d$ using the following Lemma. 

\begin{lemma}\label{lem:rho_ext}
Let $T$ be a root of the polynomial $At^d+t+1$ in an algebraic closure of the function field $\fq(A)$. Let $P_0$ and $P_\infty$ denote the zero  and the pole  of $A$ respectively. We have,  
\begin{enumerate}
    \item[(a)] $\fq(A,T)=\fq(T)$.
    \item[(b)] Both $P_0$ and $P_\infty$ are  ramified in the extension $\fq(T)/\fq(A)$. In particular, the place $P_\infty$ is totally ramified in the extension $\fq(T)/\fq(A)$, while two places lie above $P_0$, one with ramification index one and the other with ramification index $d-1$.  Moreover, 
    \begin{enumerate}
        \item[(i)] if $\gcd(q,(d-1)d) \neq 1$, then no more places ramify in the extension $\fq(T)/\fq(A)$, 
        \item[(ii)] if $\gcd(q,(d-1)d) = 1$, then exactly one more place $P$ of $\fq(A)$ is ramified in $\fq(T)/\fq(A)$, namely the place $P$ corresponding to the polynomial $A-(-1)^d(d-1)^{d-1}/d^d.$ In this case, there is one place of $\fq(T)$ that lies over $P$ with ramification index two.
    \end{enumerate}
    \end{enumerate}
\end{lemma}
\begin{proof}\noindent
\begin{enumerate}
    \item[(a)] Since $A=-(1+T)/T^d \in \fq(T)$, it is clear that $\fq(A,T)=\fq(T)$.
    \item[(b)] The pole of $A$ has only one place lying above it in $\fq(T)$, namely the zero of $T$ with ramification index $d$. The zero of $A$ has exactly two places lying above it, namely the pole of $T$ with ramification index $d-1$ and the zero of $T+1$ with ramification index one. 
If $P$ is any other place of $\fq(A)$ ramified in $\fq(T)/\fq(A)$, then $P$ corresponds to a nonzero value of $A$ such that the polynomial $At^d+t+1$ has a multiple root. In this case, the polynomial $At^d + t + 1$ and its formal derivative $dAt^{d-1}+1$ must have a common root. If $\gcd (q, d) \neq 1$, or equivalently $\char \ \fq \mid d$, then the latter has no roots. Moreover, a common root of $At^d + t + 1$ and $dAt^{d-1}+1$ is also a root of $d(At^d+t+1)-t(dAt^{d-1}+1)=(d-1)t+d$. However, if $\char \ \fq \mid d-1$, then this is impossible. This establishes (i). 
Let us now assume that $\gcd(q,(d-1)d)=1$. Our computations above show that $T=-d/(d-1)$ is the only possible multiple root of $At^d+t+1$, and in this case we have $A=-(1+T)/T^d=(-1)^d (d-1)^{d-1}/d^d$. A direct verification shows that it is in fact a multiple root. It follows that the place $P$ of $\fq(A)$ corresponding to the polynomial $A-(-1)^d(d-1)^{d-1}/d^d$ is the only place that is ramified in the extension $\fq(T)/\fq(A)$, as desired. Since the second derivative of $At^d+t+1$ only has zero as a root, it follows that the multiplicity of the root $T=-d/(d-1)$, and hence the ramification index, is two. 
\end{enumerate}
This completes the proof. 
\end{proof}

We now understand the function field extension obtained by adding one root of the polynomial $At^d+t+1$ to the rational function field $\fq(A)$. Next, we study what happens if we add two roots.

\begin{proposition}\label{prop:splitfield2}
Let $T_1$ and $T_2$ be two distinct roots of the polynomial $At^d+t+1$ in an algebraic closure of the function field $\fq(A)$. Then $\fq(A,T_1,T_2)=\fq(\rho)$, where $\rho=T_2/T_1$. Moreover, 
\begin{equation}\label{eq:T1inrho}
T_1=-\frac{\rho^{d-1}+\cdots + \rho+1}{\rho^{d-1}+\cdots+\rho}=-\frac{\rho^d-1}{\rho^d-\rho},
\end{equation}
\begin{equation}\label{eq:T2inrho}
T_2=T_1\cdot \rho=-\frac{\rho^{d-1}+\cdots + \rho+1}{\rho^{d-2}+\cdots+1}=-\frac{\rho^d-1}{\rho^{d-1}-1},
\end{equation}
and
\begin{equation}\label{eq:Ainrho}
A=-\frac{T_1+1}{T_1^d}=(-1)^d \frac{(\rho-1)(\rho^d-\rho)^{d-1}}{(\rho^d-1)^d}=(-1)^d\frac{\rho^{d-1}(\rho^{d-2}+\cdots + \rho+1)^{d-1}}{(\rho^{d-1}+\cdots + \rho+1)^d}.
\end{equation}
In particular, $\fq$ is the full constant field of $\fq(\rho)$ and $[\fq(\rho):\fq(A)]=d(d-1).$
\end{proposition}
\begin{proof}
When viewed as a polynomial in $\fq(T_1)[t]$, we see that
\begin{eqnarray*}
At^d+t+1 & = & A(t^d-T_1^d)+(t-T_1) +1+AT_1^d+T_1 \\
         & = & A(t^d-T_1^d)+(t-T_1)\\
         & = & (t-T_1)\left( A(t^{d-1}+T_1t^{d-2}+\cdots+T_1^{d-2}t+T_1^{d-1})+1\right)\\
         & = & (t-T_1)\cdot \frac{-T_1-1}{T_1} \cdot \left( \left(\frac{t}{T_1}\right)^{d-1}+\cdots \left(\frac{t}{T_1}\right)+1+\frac{-T_1}{T_1+1}\right)\\
         & = & \frac{-T_1-1}{T_1} \cdot (t-T_1)\cdot \left( \left(\frac{t}{T_1}\right)^{d-1}+\cdots \left(\frac{t}{T_1}\right)+\frac{1}{T_1+1}\right).
\end{eqnarray*}
Replacing $t$ by $T_2$ in the above equation, we have
$$\rho^{d-1} + \dots + \rho + \frac{1}{T_1 + 1} = 0,$$
which immediately proves 
Equation \eqref{eq:T1inrho}. Equations \eqref{eq:T2inrho} and \eqref{eq:Ainrho} now follow trivially from \eqref{eq:T1inrho}. Thus $A, T_1, T_2 \in \fq(\rho)$ and consequently $\fq(A,T_1,T_2)=\fq(\rho)$. Thus $\fq$ is the full constant field of $\fq(A,T_1,T_2)$. Finally, Equation \eqref{eq:Ainrho} implies $[\fq(\rho):\fq(A)]=(d-1)d$.
\end{proof}
We now look at the particular case when $d = 3$. The following Corollary is an immediate consequence of Proposition \ref{prop:splitfield2}.
\begin{corollary}\label{cor:F3}
The splitting field $F_3$ of the polynomial $At^3+t+1 \in \fq(A)[t]$ is the rational function field $\fq(\rho).$ In particular, the Galois group of $At^3+t+1$ is isomorphic to the symmetric group $S_3.$
\end{corollary}

\subsubsection{The case $\gcd(q,(d-1)d)=1$}

For $d>3$, the Galois group of the polynomial $At^d+t+1 \in \fq(A)[t]$ need not be the symmetric group $S_d$, in general. However, the next result shows that this is often the case. 

\begin{theorem}\label{thm:genus}
If $\gcd(q,(d-1)d)=1$, then the Galois group of  $At^d+t+1 \in \fq(A)[t]$ is isomorphic to the symmetric group $S_d$. Moreover, in this case, the splitting field $F_d$ of $At^d+t+1$ has full constant field $\fq$ and its genus $g_d$ is given by 
$$g_d=1+\frac{d^2-5d+2}{4}(d-2)!.$$ 
\end{theorem}
\begin{proof}
We denote by $G$ and $\overline{G}$ the Galois groups of the extensions $F_d/\fq(A)$ and $\overline{\mathbb{F}}_qF_d/\overline{\mathbb{F}}_q(A)$ respectively. 
From Proposition \ref{prop:splitfield2} we have $[\overline{\mathbb{F}}_q(A,T_1,T_2) : \overline{\mathbb{F}}_q(A,T_1)] = d-1$. Thus the Galois group $\overline{G}$ acts $2$-transitively on the $d$ roots of $At^d+t+1$. 
Let $P'$ be the place of $\overline{\mathbb{F}}_q(A)$ corresponding to $A = (-1)^d(d-1)/d^d$ in the extension $\overline{\mathbb{F}}_q(A,T_1)/\overline{\mathbb{F}}_q(A)$. Proceeding in the same way as for $P$ in Lemma \ref{lem:rho_ext}, we may conclude that there is one place above $P'$ with ramification index two while all the other places above $P'$ are unramified. By the ``Cycle Lemma'' in \cite[Section 19]{Abhyankar}, we see that $\overline{G}$ contains a transposition. Note that all the conditions in the Cycle Lemma are satisfied because we extended the constant field to $\overline{\mathbb{F}}_q$.
Thus, $\overline{G}$ is isomorphic to a 2-transitive subgroup of $S_d$ containing one transposition and, consequently, isomorphic to the symmetric group $S_d$. In particular, we see that $[\overline{\mathbb{F}}_qF_d:\overline{\mathbb{F}}_q(A)]=d!$. Since $\overline{G}$ is isomorphic to a subgroup of $G$, and $G$ is a subgroup of $S_d$, it follows that $[F_d:\fq(A)]=d!$. Moreover, the equality $[\overline{\mathbb{F}}_qF_d:\overline{\mathbb{F}}_q(A)]=[F_d:\fq(A)]$ implies that $\fq$ is the full constant field of $F_d$.   

Since $\gcd(q,(d-1)d)=1$, Lemma \ref{lem:rho_ext} implies that all the ramification in $\fq(T_1)/\fq(A)$ is tame. Since $F_d$ is the compositum of conjugates of $\fq(T_1)$, Abhyankar's lemma on ramification in compositum of function fields implies that only the places $P_\infty$, $P_0$, and $P$ ramify in $F_d/\fq(A)$ with ramification indices $d$, $d-1$ and $2$ respectively. Hence the degree $D$ of the different of the extension $F_d/\fq(A)$ equals 
$$D=\left(\frac{d-1}{d}+\frac{d-2}{d-1}+\frac{1}{2}\right)\cdot d!.$$
The stated formula for the genus of $F_d$ now follows from the Riemann-Hurwitz genus formula applied to the extension $F_d/\fq(A)$.
\end{proof}

\begin{remark}\normalfont
In the proof above, the reasoning that shows that the Galois group of the polynomial $At^d+t+1$ is the symmetric group $S_d$ is similar to that in an example from Chapter 4 in \cite{Serre}. There, it was shown that the polynomial $X^n-X^{n-1}-T \in \mathbb{Q}(T)[X]$ has Galois group $S_n$, also by observing that the Galois group is two-transitive and contains a transposition. Considering the reciprocal polynomial, one obtains that the polynomial $TX^n+X-1 \in \mathbb{Q}(T)[X]$ has Galois group $S_n$. Our polynomial would therefore have Galois group $S_d$ when viewed as a polynomial in $\mathbb{Q}(A)[t]$. The cited Cycle Lemma from \cite{Abhyankar} makes sure that the reasoning is also valid when working over $\fq$ as long as $\gcd(q,(d-1)d)=1$.  
\end{remark}

We are now ready to state and prove a sufficient condition for the existence of $A \in \fq$ such that the polynomial $At^d+t+1$ splits over $\fq$. We record this as a corollary. 

\begin{corollary}\label{cor:Fdbound}
Suppose that $\gcd(q,(d-1)d)=1$. Then there exists $A \in \fq$ such that the polynomial $At^d+t+1$ splits over $\fq$ if
\begin{equation}\label{eq:existsAsplit1}
q+1-\lfloor 2 \sqrt{q} \rfloor \left(1+\frac{d^2-5d+2}{4}(d-2)!\right) -\left(\frac{1}{d}+\frac{1}{d-1}+\frac{1}{2}\right) d! >0.
\end{equation}
\end{corollary}

\begin{proof}
Serre's refinement of Hasse-Weil bound (\cite[Theorem 5.3.1]{Stichtenoth}) shows that the function field $F_d$ has at least $q+1-\lfloor 2 \sqrt{q} \rfloor \cdot g_d$ many rational places. But, as seen above, the number of rational places of $F_d$ lying over the three ramified places of $\fq(A)$ in the extension $F_d/\fq(A)$ is at most $(1/d+1/(d-1)+1/2)\cdot d!$. Our hypothesis implies that $F_d$ has at least one more rational place, say $P$. The rational place  $\tilde{P}$ of $\fq(A)$ lying below $P$ thus splits completely in the extension $F_d/\fq(A)$. This is equivalent to $\tilde{P}$ splitting in $\fq(T_1)/\fq(A)$ and hence equivalent to the splitting of the polynomial $A(\tilde{P})t^d+t+1$ over $\fq$.
\end{proof}


We remark that the sufficient condition obtained in Corollary \ref{cor:Fdbound} can potentially be improved. First, Serre's refinement of the Hasse-Weil bound on the number of rational places in $F_d$ may not be attained. So a more precise knowledge of the number of rational places of $F_d$ can be helpful. Next, not all places lying above the three ramified places of $\fq(A)$ in $F_d/\fq(A)$ may be rational. In Section \ref{subsec:small_d} we will use these observations to refine Corollary \ref{cor:Fdbound} for some small values of $d$ and/or $q$. Before that, we present some results for particular values of $(d, q)$ with $\gcd(q,(d-1)d)>1$. 


\subsubsection{Some observations for the case $\gcd(q,(d-1)d) > 1$.}

Having derived a sufficient condition for existence of a polynomial $At^d + t + 1 \in \fq[t]$ that splits over $\fq$ under the assumption $\gcd (q, d(d-1)) = 1$, we now proceed to the case when $\gcd(q,(d-1)d) > 1$. Under the latter hypothesis, we will mainly restrict our attention to two special cases. First, note that the Galois group of $At^d + t + 1 \in \fq(A)[t]$ is the same as the Galois group of the reciprocal polynomial $t^d + t^{d-1} + A$.
\begin{enumerate}
\item[(a)]\textbf{$d = p^e$ for some $e$.} When $d$ is of the form $p^e$, where $p = \char \ \fq$, then the Galois group of such polynomials is known thanks to Abhyankar. In particular, the result \cite[Theorem 1.5]{Abhyankar2} states that the Galois group of $t^d + t^{d-1} + A$ over $\overline{\mathbb{F}}_p(A)$ is given by $\text{AGL} (1, p^e)$, that is, the group of all affine transformations of the affine line over $\mathbb{F}_{p^e}$. This group has order $p^e(p^e-1) = d(d-1)$. 

\item[(b)] \textbf{$d = p^e + 1$ for some $e$.} In this case, the reciprocal polynomial $t^d + t^{d-1} + A$ falls within the family considered in \cite[Theorem 1.3 (1.3.3)]{Abhyankar2}. This theorem states that the Galois group of $t^d + t^{d-1} + A$ over $\overline{\mathbb{F}}_p(A)$ lies between $\mathrm{PSL}(2,p^e)$ and $\mathrm{PGL}(2,p^e)$. Note that $\mathrm{PGL}(2,p^e)$ has order $d(d-1)(d-2)$, while $\mathrm{PSL}(2,p^e)$ is a subgroup of $\mathrm{PGL}(2,p^e)$ of index either one or two, depending on whether $q$ is even or odd. 
\end{enumerate}
These results show that, in these two particular cases, the splitting field of $At^d+t+1$ over the field $\overline{\mathbb{F}}_p(A)$ can be obtained by adding either two or three of its roots. In the following, we determine the splitting fields of $At^d + t + 1$ over $\fq(A)$ in the cases when $d=p^e$ and $d=p^e+1$.  We will use the notation $T_1, T_2$ and $\rho$ from Proposition \ref{prop:splitfield2}. 
\begin{theorem}\label{thm:d=pe}
If $d=p^e$, then the splitting field of $At^d+t+1$ over $\fq(A)$ is the composite of $\fq(T_1,T_2)=\fq(\rho)$ and the finite field with $p^e$ elements.  
\end{theorem}
\begin{proof}
Over $\fq(\rho)$, the equation $At^d + t + 1 = 0$ can be rewritten as an Artin-Schreier equation in the following way. 
Define
$$
B := - \frac{\rho^d - \rho}{(\rho - 1)^{d+1}} \in  \fq(\rho).
$$
From Equation \eqref{eq:Ainrho}, we obtain $A=-B^{d-1}$. Consequently,
$$
-B(At^d + t + 1) = (Bt)^d - Bt - B \in \fq(\rho)[t]. 
$$
Now the splitting field is easy to determine: By definition, $T_1$ is a root of $At^d + t + 1 = 0$ and the other roots are given by $T_1 + \alpha/B$, for some $\alpha \in \mathbb{F}_{p^e}\setminus \{0\}$. Therefore the splitting field of $At^d+t+1$ over $\fq(A)$ is precisely the composite of $\mathbb{F}_{p^e}$ and $\fq(\rho)$.
\end{proof}

As a consequence, we get a result similar to Corollary \ref{cor:Fdbound}.

\begin{corollary}\label{cor:d=pe}
    If $d=p^e$, then there exists $A \in \fq \setminus \{0\}$ such that $At^d + t + 1$ splits over $\fq$ if and only if $\mathbb{F}_{p^e} \subseteq \fq$ and $\left[ \fq : \mathbb{F}_{p^e}\right] > 1$.
\end{corollary}

\begin{proof}
If $\mathbb{F}_{p^e} \not\subseteq \fq$, then  Theorem \ref{thm:d=pe} shows that the constant field of  $F_d$ is strictly larger than $\fq$. In particular, no $\fq$-rational place of $\fq(A)$ can split in the extension $F_d/\fq(A)$. Therefore we may assume from now on that $\mathbb{F}_{p^e} \subseteq \fq$. In particular, we may assume that $F_d=\fq(\rho)$. From Lemma \ref{lem:split1}, we see that 
the only ramified places in $\fq(\rho)/\fq(A)$ are those lying above the zero and the pole of $A$. By considering Equation \eqref{eq:Ainrho}, we conclude that there are $d$ many $\fq$-rational places with ramification index $d-1$ lying above the zero of $A$, and a single totally ramified $\fq$-rational place lying above the pole of $A$. Since, the total number of $\fq$-rational places of $\fq(\rho)$ is $q+1$, there is an $\fq$-rational place $P$ of $\fq(\rho)$ that is unramified in the extension $\fq(\rho)/\fq(A)$ if and only if 
    $$q + 1 - (d + 1) > 0,$$
or equivalently $q > d$. Moreover, if $P$ is such an $\fq$-rational place of $\fq(\rho)$ that is unramified, then the place $Q$ of $\fq(A)$ lying below $P$ splits in the extension $\fq(\rho)/\fq(A)$. 
\end{proof}

Next, we investigate the case $d = p^e + 1$. Note that in this case, $d \ge 3$. Let $T_1, T_2, T_3$ be three distinct roots of $At^d + t + 1$ in an algebraic closure of $\fq(A)$. As mentioned above,  $\rho$ stands for $T_2/T_1$. We introduce a new notation $\sigma = T_3/T_1$.

\begin{theorem}\label{thm:d=pe+1}
If $d=p^e+1$, then the splitting field of $At^d+t+1$ over $\fq(A)$ is the composite of the finite field with $p^e$ elements and $\fq(T_1,T_2,T_3)=\fq((\sigma-1)/(\sigma-\rho))$.
\end{theorem}
\begin{proof}
Over $\fq(\rho)$, as in Proposition \ref{prop:splitfield2}, the polynomial $At^d + t + 1$ factors as 
\begin{eqnarray*}
At^d + t + 1 & = & A(t-T_1)\cdot \left( \frac{1}{A} + \sum_{i=1}^{d-1} t^{d-1-i}T_1^i \right)\\
& = & A(t-T_1)(t-T_2)\cdot \sum_{i=0}^{d-2}t^{d-2-i}\sum_{j=0}^iT_1^{i-j}T_2^{j}\\
& = & A(t-T_1)(t-T_2)\cdot \sum_{i=0}^{d-2}t^{d-2-i}T_1^i\sum_{j=0}^i\rho^j.
\end{eqnarray*}
The second equality is obtained by using $\frac{1}{A} + \sum_{i=1}^{d-1} T_2^{d-1-i}T_1^i = 0$.
Thus $T_3$ gives rise to a root $\sigma:=T_3/T_1$ of the polynomial $P(s):=\displaystyle{\sum_{i=0}^{d-2}s^{d-2-i}}\sum_{j=0}^i\rho^j$. Now using the fact that $d-1$ is a power of $p$, we obtain
\begin{eqnarray*}
(s-1) \cdot P(s)  =  \sum_{i=0}^{d-2}s^{d-1-i}\rho^i-\sum_{i=0}^{d-2}\rho^{i}
  =  s(s-\rho)^{d-2}-(\rho-1)^{d-2}.
\end{eqnarray*}
Moreover, 
\begin{eqnarray*}
(s-\rho) \cdot P(s)  =  \sum_{i=0}^{d-2}s^{d-1-i}-\sum_{i=0}^{d-2} \rho^{i+1}
 =  s(s-1)^{d-2}-\rho(\rho-1)^{d-2}.  
\end{eqnarray*}
Since $P(\sigma)=0$, we obtain that 
\begin{equation}\label{eq:F4}\left(\frac{\sigma-1}{\sigma-\rho}\right)^{d-2}=\frac{\sigma(\sigma-1)^{d-2}}{\sigma(\sigma-\rho)^{d-2}}=\frac{\rho(\rho-1)^{d-2}}{(\rho-1)^{d-2}}=\rho.\end{equation}
Note that Equation \eqref{eq:F4} implies that $\fq(T_1,T_2,T_3)=\fq(\rho,T_3)=\fq((\sigma-1)/(\sigma-\rho))$. 
Thus, the Galois closure of the polynomial $At^d+t+1$ over $\fq(A)$ can be obtained by adding $T_3$ to $\fq(\rho)$ and taking the compositum with the finite field with $p^e$ elements. 
\end{proof}

Again we easily derive a result analogous to Corollary \ref{cor:Fdbound}.

\begin{corollary}\label{cor:d=pe+1}
    Let $d=p^e + 1$ where $p$ is the characteristic. Then there exists $A \in \fq\setminus \{0\}$ such that $At^d + t + 1$ splits over $\fq$ if and only if $\mathbb{F}_{p^e} \subseteq \fq$ and $\left[ \fq : \mathbb{F}_{p^e}\right] > 2$.
\end{corollary}

\begin{proof}
As in the proof of Corollary \ref{cor:d=pe}, if $\mathbb{F}_{p^e} \not\subseteq \fq$, then no place of $\fq(A)$ splits in the extension $F_d/\fq(A)$, since the constant field of $F_d$ is strictly larger than $\fq$ in that case. We may therefore assume that $\mathbb{F}_{p^e} \subseteq \fq$, which implies that $F_d=\fq((\sigma-1)/(\sigma-\rho))$. 

\noindent From Lemma \ref{lem:split1}, the ramified places in $\fq((\sigma-1)/(\sigma-\rho))/\fq(A)$ are those lying above the zero and the pole of $A$. There are $d$ such $\fq$-rational places with ramification index $(d-1)(d-2)$ lying above the zero of $A$, while $(d-1)(d-2)$  such $\fq$-rational places with ramification index $d$ lying above the pole of $A$. The places above the pole are $\fq$-rational if and only if $\mathbb{F}_{p^{2e}} \subseteq \fq$. Simple calculations now show that there is an unramified $\fq$-rational place in $\fq((\sigma-1)/(\sigma-\rho))$ if and only if $q > p^{2e}$, and the result follows.
\end{proof}

\subsection{Results for $d \in \{3,4,5,6\}$}\label{subsec:small_d}

Now we are ready to state the results for small values of $d$ using our approach. We begin with the case when $d=3.$

\begin{lemma}\label{lem:d=3}
The polynomial $At^3+t+1 \in \fq[t]$ splits over $\fq$ for exactly $\lfloor (q-2)/6 \rfloor$ values of $A\in \fq \setminus \{0\}$. 
\end{lemma}
\begin{proof}
As before, let $F_3$ denote the Galois closure of the polynomial $At^3+t+1$ over $\fq(A)$. Furthermore, let $N_3$ denote the number of $A \in \fq$ for which the polynomial $At^3+t+1$ splits over $\fq$.
From Corollary \ref{cor:F3}, Theorem \ref{thm:d=pe} and \ref{thm:d=pe+1} we see that  $F_3=\fq(\rho)$. In particular, $F_3$ is a function field of genus zero. Therefore $F_3$ has $q+1$ many $\fq$-rational places. Moreover, for $d=3$, Equation \eqref{eq:Ainrho} shows that 
\begin{equation}\label{eq:AST}
A=-\frac{\rho^2 \left(\rho+1 \right)^2}{\left(\rho^2+\rho+1\right)^3}.
\end{equation}
Using this equation, a direct calculation shows that if $\char \ \fq = 3$, then the pole of $A$ is totally ramified with the zero of $\rho-1$ lying above it.  If $q \equiv 1 \pmod3$, two rational places lie above it, namely the zeroes of $\rho^2+\rho+1$. If $q \equiv 2 \pmod3$, one place of degree two lies above it namely the (in this case non-rational) zero of $\rho^2+\rho+1$. The zero of $A$ has three rational places lying above it in the extension $\fq(\rho)/\fq(A)$, namely the zero and pole of $\rho$ as well as the zero of $\rho+1$. If the characteristic of $\fq$ is two or three, then no more places ramify in $\fq(\rho)/\fq(A)$. Otherwise, the zero of $A+4/27$ is ramified as well in accordance with Lemma \ref{lem:split1}. One can directly verify, using Equation \eqref{eq:AST}, that three rational places lie above this place, namely the zeroes of $\rho-1$, $\rho+1/2$ and $\rho+2$. Since apart from the ramified places, the Galois group of the extension $\fq(\rho)/\fq(A)$ acts with orbits of length six on the remaining $\fq$-rational places of $\fq(\rho)$, the number of the sought $A$ is equal to one-sixth of the number of rational places of $\fq(\rho)$ that do not ramify in the extension $\fq(\rho)/\fq(A)$. Now using the above and distinguishing congruence classes of $q$ modulo six, we find that
$$N_3 = \left\{ 
\begin{array}{rl}
\frac{q-3}{6}  & \text{if $q \equiv 0 \pmod{3}$,}  \\[5pt]
\frac{q-7}{6}  & \text{if $q \equiv 1 \pmod{6}$,}  \\[5pt]
\frac{q-2}{6}  & \text{if $q \equiv 2 \pmod{6}$,}  \\[5pt]
\frac{q-4}{6}  & \text{if $q \equiv 4 \pmod{6}$,}  \\[5pt]
\frac{q-5}{6}  & \text{if $q \equiv 5 \pmod{6}$.} 
\end{array}
\right.$$ 
This establishes the assertion and completes the proof. 
\end{proof}

\begin{theorem}\label{thm:d=3}
For $q \ge 3$, there exists an absolutely irreducible cubic curve defined over $\fqs$ that intersects $\cH_q$ in $3(q+1)$ many distinct $\fqs$-rational points.
\end{theorem}
\begin{proof}
First, assume that $q\ge 8$. From Lemma \ref{lem:d=3}, there exists $A\in \fq \setminus \{0\}$ such that $At^3+t+1 \in \fq[t]$ splits over $\fq$. The assertion now follows from Lemma \ref{lem:split1}. 

When $3 \le q \le 7$, we consider the polynomial  
$$
f(X, Y, Z) := 
\begin{cases}
    X^3+Y^3+Z^3+ XY^2+X^2Z - YZ^2, \ \ &\text{if} \ \ q = 3 \\
    X^3+Y^3+Z^3+XY^2+X^2Z + YZ^2+ XZ^2, \ \ &\text{if} \ \ q = 4 \\
    X^3+Z^3-Y^2Z, \ \ &\text{if} \ \ q=5 \\
    X^3 +4XY^2 + YZ^2,   \ \ &\text{if} \ \ q = 7.
\end{cases}
$$

It can then be checked computationally that the cubic given by the equation $f(X, Y, Z) = 0$ satisfy the desired property.
This completes the proof. 
\end{proof}

Having settled the case $d=3$ completely, we now consider the case $d=4$.

\begin{lemma}\label{lem:d=4}
Let $N_4$ denote the number of $A \in \fq \setminus \{0\}$ for which the polynomial $At^4+t+1$ splits over $\fq$. Then 
$$
N_4 = \left\{ 
    \begin{array}{cl}
        0 & \text{ if } q = 2^e \text{ and } e \text{ is odd,} \\[5pt]
        \frac{q-4}{12} & \text{ if } q = 2^e \text{ and } e \text{ is even,} \\[5pt]
        \frac{q+1}{24} & \text{ if } q\equiv 23 \pmod {24} \text{ and } \\[5pt]
        \left\lfloor\frac{q-2}{24}\right\rfloor & \text{ otherwise.} 
    \end{array}
\right.
$$ 
\end{lemma}
\begin{proof}
First, Theorems \ref{thm:genus}, \ref{thm:d=pe} and \ref{thm:d=pe+1} imply that $F_4$ has genus zero. 
Now, the given expression for $N_4$ follows from a case-by-case analysis. 

We consider the case $q=2^e$ first. Using Corollary \ref{cor:d=pe} we immediately get the result when $e$ is odd. When $e$ is even, Theorem \ref{thm:d=pe} implies that $F_d = \fq(\rho)$, so the degree of the extension $F_4/\fq(A)$ is 12. The pole of $A$ is totally ramified and there are four rational places above the zero of $A$, each with ramification index three. Hence, $\fq(\rho)$ has exactly five ramified rational places, and the result now follows by a similar reasoning as in Lemma \ref{lem:d=3}.

If $q$ is odd, then the extension $F_4/\fq(A)$ has degree $24$. The rationality of the ramified places in $F_4/\fq(A)$ can be determined by studying what happens in the extension $\fq(\rho)/\fq(A)$. When $d= 4$, then Equation \eqref{eq:Ainrho} is equivalent to $$A=\rho^3(\rho^2+\rho+1)^3/(\rho^3+\rho^2+\rho+1)^4.$$ 
The places of $F_4$ above the zero of $A$ are all rational if and only if the polynomial $\rho^2+\rho+1$ splits over $\fq$. Indeed, if $\rho^2+\rho+1$ splits over $\fq$, then the pole of $T_1$ in $\fq(T_1)$, which lies above the zero of $A$, splits in $\fq(\rho)/\fq(T_1)$ and hence in $F_4/\fq(\rho)$. On the other hand, if it does not split, then there is a place of $\fq(\rho)$ of degree two above the zero of $A$. Note that $\rho^2+\rho+1$ splits over $\fq$ if and only if $3 \mid q-1$, that is $q \equiv 1 \pmod 3$. 

The pole of $A$ is totally ramified in $\fq(T_1)/\fq(A)$. Using the factorization $\rho^3+\rho^2+\rho+1=(\rho+1)(\rho^2+1)$, we see that $\rho^3+\rho^2+\rho+1$ splits in $\fq(\rho)/\fq(T_1)$ if and only if the polynomial $\rho^2+1$ splits over $\fq$. Consequently, $\rho^3+\rho^2+\rho+1$ splits over $\fq$ if and only if $q \equiv 1 \pmod{4}$.

If $\char \ \fq \neq 2, 3$, then the zero of $A-27/256$ has three places of $\fq(T_1)$ lying above it. The only ramified place is the zero of $T_1+4/3$. This place splits in $\fq(\rho)/\fq(T_1)$ if and only if the polynomial $\rho^2+2\rho+3=(\rho+1)^2+2$ splits over $\fq$. This happens if and only if $q \equiv 1 \pmod{8}$ or $q \equiv 3 \pmod{8}$.
Now, the assertion follows using similar steps as in the proof of Lemma \ref{lem:d=3}.
\end{proof}

The lemma immediately gives the following theorem. 

\begin{theorem}\label{thm:d=4}
Suppose that either $q \in \{16,23\}$ or $q \ge 27$ is a prime power but not an odd power of two. Then there exists an absolutely irreducible quartic curve defined over $\fqs$ that intersects  $\cH_q$ in $4(q+1)$ distinct $\fqs$-rational points.
\end{theorem}

\begin{remark}\normalfont
     The results from the previous sections also imply the existence of a quartic with this property for $q \in \{ 2, 3, 4, 7, 8\}$. For $q \in \{5, 9, 11, 13, 17, 19, 25\}$ one can choose the quartic given by $f(X,Y,Z) = 0$ with
     
     $$
        f(X, Y, Z) := 
            \begin{cases}
                X^3Y + 2Y^2Z^2 + Z^4, \ \ &\text{if} \ \ q = 5 \\
                X^4 + Y^3Z - Y^2Z^2 + YZ^3, \ \ &\text{if} \ \ q = 9 \\
                X^4 - Y^4 - \omega^{16}Z^4, \ \ &\text{if} \ \ q=11 \\
                X^3Y + Y^3Z + XZ^3, \ \ &\text{if} \ \ q=13 \\
                X^4 + 13Y^3Z + 14Y^3Z^2, \ \ &\text{if} \ \ q=17 \\
                X^4 - \omega^4Y^4 - \omega^{24} Z^4, \ \ &\text{if} \ \ q=19 \\
                X^2Y^2 + X^2Z^2 + Y^2Z^2,   \ \ &\text{if} \ \ q = 25,
            \end{cases}
    $$
    \\
    where $\omega$ is a primitive element of $\fqs$. Combining this with Theorem \ref{thm:d=4}, we see that the case $d=4$ is settled, except if $q>8$ is an odd power of two.
\end{remark}

We conclude this article by briefly outlining the scenario in the case when $d = 5$ and $d =6$. The results from the previous sections imply the existence of an irreducible quintic intersecting the Hermitian in $5(q+1)$ distinct $\fqs$-rational points for $q \in \{3,4,9\}$. Moreover, Corollary \ref{cor:Fdbound} applies whenever $q > 233$ and $\gcd(q, 20) = 1$. A direct check with a computer shows that this bound can be improved to $q > 131$; we find that $q = 131$ is the largest value of $q$, with $\gcd(q,20)=1$, for which the polynomial never splits. There are several values of $q < 131$ satisfying $\gcd(q,20) = 1$, for which the polynomial splits for at least one choice of $A$. In fact, this happens for 
$$
    q \in \{67, 79, 83, 101, 103, 107, 109, 113, 121, 127\}.  
$$
For $q = 5^e$, it follows from Corollary \ref{cor:d=pe} that a desired $A \in \fq \setminus\{0\}$ exists if and only if $e > 1$. Similarly, we conclude from Corollary \ref{cor:d=pe+1} that a suitable $A \in \fq\setminus\{0\}$ exists for $q = 2^e$ exactly when $e$ is even and $e>2$. Hence the case $d=5$ is settled except for finitely many small values of $q$ and except for the case where $q$ is an odd power of two.
However, counting the number of $A\in \fq\setminus\{0\}$ for which the polynomial $At^d + t + 1$ splits is more complicated now. One of the main reasons for this is that the splitting behavior above the place corresponding to $T_1 = -d/(d-1)$ in the extension $\fq(\rho)/\fq(T_1)$ depends on the splitting of the polynomial $t^3 + 2t^2 + 3t + 4 \in \fq[t]$. The corresponding polynomials were easier to handle for $d = 3$ and $d=4$ since they were of degree one and two, respectively. Another reason is that the splitting field is not always rational any more. Indeed, by Theorem \ref{thm:genus}, the genus is $4$ when $\gcd(q, 20) = 1$, so the exact total number of $\fq$-rational places is harder to determine. 

For $d = 6$ the results of the previous sections guarantee a positive answer for $q\in \{3,4,5,11\}$. Besides from this, we remark that Corollary \ref{cor:Fdbound} applies for $q > 10766$ and that calculations with a computer show that this bound can be lowered to $q > 1877$. More generally, in the cases where Corollary \ref{cor:d=pe} and \ref{cor:d=pe+1} do not apply, some more work would be needed to determine exactly what happens when $\gcd(q, d(d-1)) > 1$, but this is left as future work. 

\section*{Acknowledgments}
This work was supported by a research grant (VIL”52303”) from Villum Fonden. Most of the work towards this project was done during a collaboration visit by Mrinmoy Datta to the Technical University of Denmark. He sincerely thanks the university for providing the opportunity to visit, warm hospitality, and the great environment supporting the work. Moreover, Mrinmoy Datta is partially supported by  the grant DST/INT/RUS/RSF/P-41/2021 from the Department of Science and Technology, Govt. of India and the grant SRG/2021/001177 from Science and Engineering Research Board, Govt. of India.

\end{document}